\documentclass[final]{siamltex}
\usepackage{graphicx, xcolor}
\usepackage{amssymb}
\usepackage{subfigure}
\usepackage{caption}
\usepackage{amsfonts}
\usepackage{amsmath}
\usepackage{enumerate}

\numberwithin{equation}{section}

\newcommand{\ie}{I_{\epsilon}(x)}
\newcommand{\mathp}{\mathcal{P}}

\title{Spectral analysis of the truncated Hilbert transform with overlap}
\author{Reema Al-Aifari\footnotemark[1]
\and Alexander Katsevich\footnotemark[2]}

\begin{document}
\maketitle

\footnotetext[1]{Department of Mathematics, Vrije Universiteit Brussel, Brussels B-1050, Belgium}
\footnotetext[2]{Department of Mathematics, University of Central Florida, FL 32816, USA}

\begin{abstract}
We study a restriction of the Hilbert transform as an operator $H_T$ from $L^2(a_2,a_4)$ to $L^2(a_1,a_3)$ for real numbers $a_1 < a_2 < a_3 < a_4$. The operator $H_T$ arises in tomographic reconstruction from limited data, more precisely in the method of differentiated back-projection (DBP). There, the reconstruction requires recovering a family of one-dimensional functions $f$ supported on compact intervals $[a_2,a_4]$ from its Hilbert transform measured on intervals $[a_1,a_3]$ that might only overlap, but not cover $[a_2,a_4]$. We show that the inversion of $H_T$ is ill-posed, which is why we investigate the spectral properties of $H_T$.

We relate the operator $H_T$ to a self-adjoint \textit{two-interval Sturm-Liouville problem}, for which we prove that the spectrum is discrete. The Sturm-Liouville operator is found to commute with $H_T$, which then implies that the spectrum of $H_T^* H_T$ is discrete. Furthermore, we express the singular value decomposition of $H_T$ in terms of the solutions to the Sturm-Liouville problem. The singular values of $H_T$ accumulate at both $0$ and $1$, implying that $H_T$ is not a compact operator. We conclude by illustrating the properties obtained for $H_T$ numerically. 
\end{abstract}

\maketitle

\section{Introduction}\label{intro}
In tomographic imaging, which is widely used for medical applications, a 2D or 3D object is illuminated by a penetrating beam (usually X-rays) from multiple directions, and the projections of the object are recorded by a detector. Then one seeks to reconstruct the full 2D or 3D structure from this collection of projections. When the beams are sufficiently wide to fully embrace the object and when the beams from a sufficiently dense set of directions around the object can be used, this problem and its solution are well understood \cite{nat}. When the data are more limited, e.g. when only a reduced range of directions can be used or only a part of the object can be illuminated, the image reconstruction problem becomes much more challenging. 

Reconstruction from limited data requires the identification of specific subsets of line integrals that allow for an exact and stable reconstruction. One class of such configurations that have already been identified, relies on the reduction of the 2D and 3D reconstruction problem to a family of 1D problems. The Radon transform can be related to the 1D Hilbert transform along certain lines by differentiation and back-projection of the Radon transform data (\textit{differentiated back-projection} or DBP). Inversion of the Hilbert transform along a family of lines covering a sub-region of the object (\textit{region of interest} or ROI) then allows for the reconstruction within the ROI.

This method goes back to a result by Gelfand and Graev \cite{gg}. Its application to tomography was formulated by Finch \cite{finch} and was later made explicit for 2D in \cite{noo, yyww, zps} and for 3D in \cite{pnc, yzyw, zhuang, zp}. To reconstruct from data obtained by the DBP method, it is necessary to solve a family of 1D problems which consist of inverting the Hilbert transform data on a finite segment of the line. If the Hilbert transform $H f$ of a 1D function $f$ was given on all of $\mathbb{R}$, then the inversion would be trivial, since $H^{-1} = - H$. In case $f$ is compactly supported, it can be reconstructed even if $H f$ is not known on all of $\mathbb{R}$. Due to an explicit reconstruction formula by Tricomi \cite{tricomi}, $f$ can be found from measuring $H f$ only on an interval that covers the support of $f$.
However, a limited field of view might result in configurations in which the Hilbert transform is known only on a segment that does not completely cover the object support. One example of such a configuration is known as the interior problem \cite{cour, ekat, kcnd, yyw}. Given real numbers $a_1 < a_2 < a_3 < a_4$, the interior problem corresponds to the case in which the Hilbert transform of a function supported on $[a_1,a_4]$ is measured on the smaller interval $[a_2,a_3]$. 

In this paper, we study a different configuration, namely supp $f = [a_2,a_4]$ and the Hilbert transform is measured on $[a_1,a_3]$. We will refer to this configuration as the truncated problem \textit{with overlap}: the operator $H_T$ we consider is given by $\mathp_{[a_1,a_3]} H \mathp_{[a_2,a_4]}$, where $H$ is the usual Hilbert transform acting on $\mathcal{L}^2(\mathbb{R})$, and $\mathp_{\Omega}$ stands for the projection operator $(\mathp_{\Omega} f) (x) = f(x)$ if $x \in \Omega$, $(\mathp_{\Omega} f)(x) = 0$ otherwise. For finite intervals $\Omega_1$, $\Omega_2$ on $\mathbb{R}$, the \textit{interior} problem corresponds to $\mathp_{\Omega_1} H \mathp_{\Omega_2}$ for $\Omega_1 \subset \Omega_2$. The truncated Hilbert transform \textit{with a gap} occurs when the intervals $\Omega_1$ and $\Omega_2$ are separated by a gap, as in \cite{kat1}. Figure \ref{fig:ht} shows the different setups. Examples of configurations in which the truncated Hilbert transform \textit{with overlap} and the interior problem occur are given in Figures \ref{fig:tp} and \ref{fig:ip}. The truncated problem \textit{with overlap} arises for example in the "missing arm" problem. This is the case where the field of view is large enough to measure the torso but not the arms. 

\begin{figure}[ht!]
     \begin{center}
     
        \subfigure[Finite Hilbert transform.]{
            \label{fig:fht}
            \includegraphics[width=0.45\textwidth]{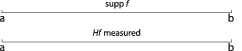}
        }
         \subfigure[Interior problem.]{
            \label{fig:iht}
            \includegraphics[width=0.45\textwidth]{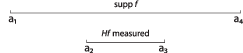}
        }\\
        \subfigure[Truncated Hilbert transform \textit{with overlap}.]{
           \label{fig:thto}
           \includegraphics[width=0.45\textwidth]{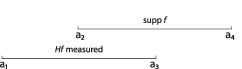}
        }
        \subfigure[Truncated Hilbert transform \textit{with a gap}.]{
           \label{fig:thtg}
           \includegraphics[width=0.45\textwidth]{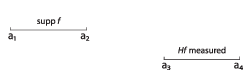}
        }
 \end{center}
    \caption{
       Different setups for $\mathp_{\Omega_1} H \mathp_{\Omega_2}$. The upper interval shows the support $\Omega_2$ of the function $f$ to be reconstructed. The lower interval is the interval $\Omega_1$ where measurements of the Hilbert transform $H f$ are taken. This paper investigates case (c).
     }
   \label{fig:ht}
\end{figure}

\begin{figure}[ht!]
     \begin{center}
        \subfigure[]{
            \label{fig:tp1}
            \includegraphics[width=0.45\textwidth]{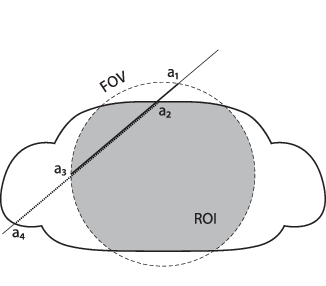}
        }
        \subfigure[]{
           \label{fig:tp2}
           \includegraphics[width=0.45\textwidth]{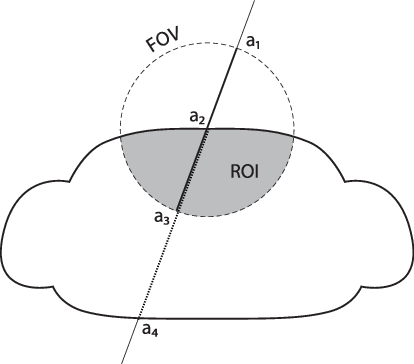}
        }
    \end{center}
     \caption{
       Two examples of the truncated problem \textit{with overlap}, Fig. \ref{fig:tp1} shows the missing arm problem. In both cases, the field of view (FOV) does not cover the object support. On the line intersecting the object, measurements can only be taken within the FOV, i.e. from $a_1$ to $a_3$. The Hilbert transform is not measured on $[a_3,a_4]$. Consequently, a reconstruction can only be aimed at in the grey-shaded intersection of the FOV with the object support, called the region of interest (ROI).
     }
   \label{fig:tp}
\end{figure}

\begin{figure}[ht!]
     \begin{center}
            \includegraphics[width=0.45\textwidth]{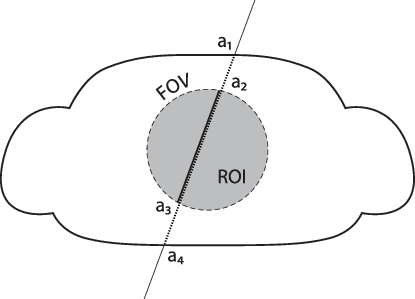}
    \end{center}
     \caption{
       The interior problem. Here, the FOV also does not cover the object support. The line intersecting the object is such that the Hilbert transform is only measured in a subinterval $[a_2,a_3]$ of the intersection $[a_1,a_4]$ of the line with the object support. The ROI is the grey-shaded intersection of the FOV with the object support. In this case stable reconstruction
of the shaded ROI is impossible unless additional information is available. 
     }
   \label{fig:ip}
\end{figure}

\clearpage

Fix any four real numbers $a_1 < a_2 < a_3 < a_4$. We define the truncated Hilbert transform \textit{with overlap} as the operator
\\
\begin{definition}
\begin{equation}\label{tht}
(H_T f)(x) := \frac{1}{\pi} \text{p.v. } \int_{a_2}^{a_4} \frac{f(y)}{y-x} dy, \quad x \in (a_1,a_3)
\end{equation}
where $p.v.$ stands for the principal value. In short,
\begin{equation*}
H_T := \mathp_{[a_1,a_3]} H \mathp_{[a_2,a_4]},
\end{equation*}
where $H$ is the ordinary Hilbert transform on $\mathcal{L}^2(\mathbb{R})$.
\end{definition}
\\

As we will prove in what follows, the inversion of $H_T$ is an ill-posed problem in the sense of Hadamard \cite{engl}. In order to find suitable regularization methods for its inversion, it is crucial to study the nature of the underlying ill-posedness, and therefore the spectrum $\sigma(H_T^* H_T)$. An important question that arises here is whether the spectrum is purely discrete. This question has been answered for similar operators before, but with two very different answers. In \cite{koppelman}, it was shown that the finite Hilbert transform defined as $H_F = \mathp_{[a,b]} H \mathp_{[a,b]}$ has a continuous spectrum $\sigma(H_F) = [-i,i]$. On the other hand, in \cite{kat2}, we find the result that for the interior problem $H_I = \mathp_{[a_2,a_3]} H \mathp_{[a_1,a_4]}$, the spectrum $\sigma(H_I^* H_I)$ is purely discrete. 

The main result of this paper is that $H_T^* H_T$ has only discrete spectrum. In addition, we obtain that $0$ and $1$ are accumulation points of the spectrum. Furthermore, we find that the singular value decomposition (SVD) of the operator $H_T$ can be related to the solutions of a Sturm-Liouville (S-L) problem. For the actual reconstruction, one would aim at finding $f$ in \eqref{tht} only within a region of interest (ROI), i.e. on $[a_2,a_3]$. A stability estimate as well as a uniqueness result for this setup were obtained by Defrise et al in \cite{defrise}. A possible method for ROI reconstruction is the truncated SVD. Thus, it is of interest to study the SVD of $H_T$ also for the development of reconstruction algorithms.

In \cite{kat1} and \cite{kat2}, singular value decompositions are obtained for the truncated Hilbert transform \textit{with a gap} $\mathp_{[a_3,a_4]} H \mathp_{[a_1,a_2]}$ and for $H_I$. This is done by relating the Hilbert transforms to differential operators that have discrete spectra. We follow this procedure, but obtain a differential operator that is different in nature. In \cite{kat1} and \cite{kat2} the discreteness of the spectra follows from standard results of singular S-L theory (see e.g. \cite{zettl}). In the case of truncated Hilbert transform \eqref{tht} we have to investigate the discreteness of the spectrum of the related differential operator explicitly.

The idea is to find a differential operator for which the eigenfunctions are the singular functions of $H_T$ on $(a_2,a_4)$. We define the differential operator similarly to the one in \cite{kat1}, \cite{kat2}, but then the question is which boundary conditions to choose in order to relate the differential operator to $H_T$. To answer this question we first develop an intuition about the singular functions of $H_T$. 

Let $\{ \sigma_n; f_n, g_n \}$ denote the singular system of $H_T$ that we want to find. The problem can be formulated as finding a complete orthonormal system $\{ f_n \}_{n \in \mathbb{N}}$ in $\mathcal{L}^2(a_2,a_4)$ and an orthonormal system $\{ g_n \}_{n \in \mathbb{N}}$ in $\mathcal{L}^2(a_1,a_3)$ such that there exist real numbers $\sigma_n$ for which
\begin{align*}
H_T f_n &= \sigma_n g_n, \\
H_T^* g_n &=\sigma_n f_n.
\end{align*}
At the moment, the $g_n$'s only have to be complete in  $\text{Ran}(H_T)$, but as we will see in Section \ref{section:acc}, $\text{Ran}(H_T)$ is dense in $\mathcal{L}^2(a_1,a_3)$.
\\
As will be shown in Section \ref{section:svd}, the functions $f_n$ and $g_n$ 
\begin{enumerate}[(a)]
\item can only be bounded or of logarithmic singularity at the points $a_i$, \\
\item do not vanish at the edges of their supports ($a_2^+$, $a_4^-$ for $f_n$, and $a_1^+$, $a_3^-$ for $g_n$).
\end{enumerate}
We will now make use of the following results from \cite{gakhov}, Sections 8.2 and 8.5:
\\
\begin{lemma}[Local properties of the Hilbert transform]\label{propH}
Let $f$ be a function with support $[b,d] \subset \mathbb{R}$. And let $c$ be in the interior of $[b,d]$.
\begin{enumerate}
\item If $f$ is H\"{o}lder continuous (for some H\"{o}lder index $\alpha$) on $[b,d]$, then close to $b$ the Hilbert transform of $f$ is given by
\begin{equation}
(H f)(x) = - \frac{1}{\pi}f(b^+) \ln |x-b| + H_0(x)
\end{equation}
where $H_0$ is bounded and continuous in a neighborhood of $b$. \\
\item If in a neighborhood of $c$, the function $f$ is of the form $f(x) = \tilde{f}(x) \ln|x-c|$ for H\"{o}lder continuous $\tilde{f}$, then close to the point $c$ its Hilbert transform is of the form
\begin{equation*}
(Hf )(x) = H_0(x),
\end{equation*}
where $H_0$ is bounded with a possible finite jump discontinuity at $c$. \\
\item If $f$ is of the form $f(x) = \tilde{f}(x) \ln|x-b|$ on $[b,c]$, where $\tilde{f}$ is H\"{o}lder continuous, then its Hilbert transform at $b$ has a singularity of the order $\ln^2|x-b|$ if $\tilde{f}(b) \neq 0$. \\
\end{enumerate}
\end{lemma}
Suppose $f_n$ has a logarithmic singularity at $a_2^+$. Since $H_T$ integrates over $[a_2,a_4]$, the function $H_T f_n$ would have a singularity at $a_2$ of order $\ln^2 |x-a_2|$. Hence, this would violate the property of $g_n$ at $a_2$. Therefore, $f_n$ has to be bounded at $a_2^+$. If $f_n$ does not vanish at $a_2^+$, this leads to logarithmic singularities of $H_T f_n$ and $g_n$ at $a_2$. Using the same argument we conclude that $g_n$ is bounded at $a_3^-$ and $f_n$ has a logarithmic singularity at $a_3$.

On the other hand, since $g_n$ is bounded at $a_3^-$, $H_T f_n$ is also bounded there. This requires that close to $a_3$, $f_n = f_{n,1} + f_{n,2} \ln|x-a_3|$ for functions $f_{n,i}$ continuous at $a_3$. A similar argument holds for $g_n$ at $a_2$. Close to that point, $g_n = g_{n,1} + g_{n,2} \ln|x-a_2|$ for functions $g_{n,i}$ continuous at $a_2$. \\
Clearly, $H_T f_n$ is bounded at $a_1^+$ and $H_T g_n$ is bounded at $a_4^-$. Therefore, $f_n$ has to be bounded at $a_4^-$ and $g_n$ must be bounded at $a_1^+$.

Thus, if we want to show the commutation of $H_T$ with a differential operator that acts on $f_n(x)$, $x\in(a_2,a_4)$, we need to impose boundary conditions at $a_2^+$ and $a_4^-$ that require boundedness and some transmission conditions at $a_3$ that make the bounded term and the term in front of the logarithm in $f_n$ continuous at $a_3$.

\begin{figure}[ht!]
     \begin{center}

        \subfigure[Sketch of $f_n$'s.]{
            \label{fig:logf}
            \includegraphics[width=0.45\textwidth]{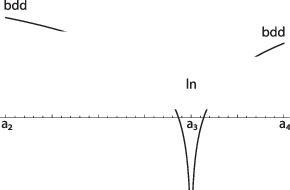}
        }
        \subfigure[Sketch of $g_n$'s.]{
           \label{fig:logg}
           \includegraphics[width=0.45\textwidth]{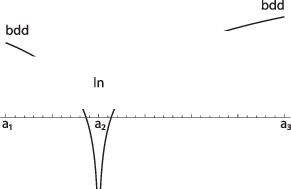}
        }
        
    \end{center}
     \caption{
      Intuition about the singular functions of $H_T$.
     }
   \label{fig:sketch}
\end{figure}

Having found these properties of the singular functions of $H_T$ (in case the SVD for $H_T$ exists), in Section \ref{do} we introduce a differential operator and find a self-adjoint extension for this operator. We then show in Section \ref{spec} that this self-adjoint differential operator $L_S$ has a discrete spectrum. In Section \ref{section:svd} we establish that $L_S$ commutes with the operator $H_T$. This allows us to find the SVD of $H_T$. In Section \ref{section:acc} we then study the accumulation points of the singular values of $H_T$. In particular, we find that $H_T$ is not a compact operator. Finally, we conclude by showing numerical examples in Section \ref{section:num}.

\section{Introducing a differential operator}\label{do}
In this section, we find two differential operators $L_S$ and $\tilde{L}_S$ that will turn out to have a commutation property of the form 
\begin{equation}\label{commutingoperators}
H_T L_S = \tilde{L}_S H_T.
\end{equation}
In order to find the SVD of $H_T$, we will be interested in finding $L_S$ and $\tilde{L}_S$ with simple discrete spectra. Initially, it is not apparent whether differential operators with such properties exist and if so, how to find them. We do not know of a coherent theory that relates certain integral operators to differential operators via a commutation property as the above. However, there have been examples of integral operators for which -- by what seems to be a lucky accident -- such differential operators exist. 

One instance is the well-known Landau-Pollak-Slepian (LPS) operator that arises in signal processing in the study of time- and bandlimited representations of signals \cite{slepian1961, landau1961, landau1962}. There, it is of interest to find the largest eigenvalue of the LPS operator $ \mathp_{[-T,T]} \mathcal{F}^{-1} \mathp_{[-W,W]} \mathcal{F} \mathp_{[-T,T]}$. Here, $\mathcal{F}$ is the Fourier transform, and $T$ and $W$ are some positive numbers. This operator happens to commute with a second order differential operator, of which the eigenfunctions and eigenvalues had been studied long before its connection to the LPS operator was known. The eigenfunctions of this differential operator are the so-called prolate spheroidal wave functions and they turn out to be the eigenfunctions of the LPS operator as well. The work of Landau, Pollak and Slepian has been generalized and extended by Gr\"{u}nbaum et al. \cite{grunbaum1982}.

More recent examples of integral operators with commuting differential operators are the interior Radon transform \cite{MaassInterior} and two instances of the truncated Hilbert transform mentioned earlier \cite{kat1,kat2}.

To start our search for $L_S$ and $\tilde{L}_S$, we follow the procedure in \cite{kat1, kat2} and define a differential operator 
\\
\begin{definition}
\begin{equation}
L(x,d_x) \psi(x):= (P(x)\psi'(x))'+2 (x-\sigma)^2 \psi(x)
\end{equation}
where
\begin{equation}
P(x) = \prod_{i=1}^4 (x-a_i), \quad  \sigma = \frac{1}{4}\sum_{i=1}^4 a_i.
\end{equation}
\end{definition}
The four points $a_i$ are all \textit{regular singular}, and in a complex neighborhood of each $a_i$ the functions $(x-a_i) \cdot P'(x)/P(x)$ and $(x-a_i)^2 \cdot 2 (x-\sigma)^2/P(x)$ are complex analytic. The term \textit{regular singular point} is standard in the general theory of differential equations and, as such, is also used in the theory of S-L equations, see e.g. \cite{teschl} for this and other terminology and basic properties of S-L equations. Consequently, by the method of Fuchs-Frobenius it follows that for $\lambda \in \mathbb{C}$ any solution of $L \psi = \lambda \psi$ is either bounded or of logarithmic singularity close to any of the points $a_i$, see \cite{teschl}. Away from the singular points $a_i$ the analyticity of the solutions follows from the analyticity of the coefficients of the differential operator $L$. More precisely, in a left and a right neighborhood of each regular singular point $a_i$, there exist two linearly independent solutions of the form
\begin{align}
\psi_1(x) &= |x-a_i|^{\alpha_1} \sum_{n=0}^{\infty} b_n (x-a_i)^n, \label{ff1} \\
\psi_2(x) &= |x-a_i|^{\alpha_2} \sum_{n=0}^{\infty} d_n (x-a_i)^n + k \ln |x-a_i| \psi_1(x), \label{ff2}
\end{align}
where without loss of generality we can assume $b_0=d_0=1$. The exponents $\alpha_1$ and $\alpha_2$ are the solutions of the indicial equation
\begin{equation*}
\alpha^2 + (p_0-1) \alpha + q_0 = 0,
\end{equation*}
where 
\begin{align}
p_0 &= \lim_{x \to a_i} (x-a_i) P'(x)/P(x), \\
q_0 &= \lim_{x \to a_i} (x-a_i)^2 [2 (x-\sigma)^2-\lambda]/P(x).
\end{align}
With our choice of $P$, this gives $\alpha_1 = \alpha_2 = 0$ which implies $k \neq 0$. For the bounded solution in \eqref{ff1}, $\alpha_1=0$ results in $\psi_1(a_i) \neq 0$. The radius of convergence of the series in \eqref{ff1} and \eqref{ff2} is the distance to the closest singular point different from $a_i$.  
In a left and in a right neighborhood of $a_i$, the general form of the solutions of $(L-\lambda) \psi = 0$ is
\begin{align}
\psi_1(x) &= \ell_{0} \sum_{n=0}^{\infty} b_n (x-a_i)^n \label{dof1} \\
\psi_2(x) &=  \ell_{1} \sum_{n=0}^{\infty} d_n (x-a_i)^n + \ell_{2} \ln |x-a_i| \sum_{n=0}^{\infty} b_n (x-a_i)^n \label{dof2}
\end{align}
for some constants $\ell_{j}$. Hence we have one degree of freedom for the bounded solution, and two -- for the unbounded solution. Clearly, for the bounded solutions \eqref{dof1}, the coefficients $b_n$ are the same on both sides of $a_i$, since we have assumed $b_0=1$. However, the bounded part of the unbounded solutions \eqref{dof2} may have different coefficients $d^-_n$ and $d^+_n$ to the left and to the right of $a_i$ respectively.

\subsection{The Maximal and Minimal Domains and Self-Adjoint Realizations}

Since we are interested in a differential operator that commutes (on some set to be defined) with $H_T$, we want to consider $L$ on the interval $(a_2,a_4)$. Due to the regular singular point $a_3$ in the \textit{interior} of the interval, standard techniques for singular S-L problems are not applicable. It is crucial for our application that we identify a commuting \textit{self-adjoint} operator, for which the spectral theorem can be applied. We therefore wish to study all self-adjoint realizations; we follow the treatment in Chapter 13 in \cite{zettl} which gives a characterization of all self-adjoint realizations for \textit{two-interval problems}, of which problems with an interior singular point are a special case.

First of all, one needs to define the maximal and minimal domains on $I_j=(a_j,a_{j+1})$ (see Chapter 9 in \cite{zettl}). Let $AC_{loc}(I)$ be the set of all functions that are absolutely continuous on all compact subintervals of the open interval $I$. Then, 
\begin{align}
D_{j,\max} &:= \{ \psi: I_j \rightarrow \mathbb{C}: \psi, P \psi' \in AC_{loc}(I_j); \psi, L \psi \in \mathcal{L}^2(I_j) \}, \\
D_{j,\min} &:= \{ \psi \in D_{j,\max}: \text{supp }\psi\subset (a_j,a_{j+1}) \},
\end{align}
and the related maximal and minimal operators are defined as follows:
\begin{align}
L_{j,\max} &:= L (D_{j,\max}): D_{j,\max} \rightarrow \mathcal{L}^2(I_j), \\
L_{j,\min} &:= L (D_{j,\min}): D_{j,\min} \rightarrow \mathcal{L}^2(I_j).
\end{align}
We shall follow essentially the procedure in Chapter 13 in \cite{zettl}, to which we refer for more details. On $(a_2,a_4)$, the maximal and minimal domains and the corresponding operators are defined as the direct sums:
\\
\begin{definition}
The maximal and minimal domains $D_{\max}$, $D_{\min} \subset \mathcal{L}^2(a_2,a_4)$ and the operators $L_{\max}$, $L_{\min}$ are defined as
\begin{align*}
D_{\max} &:= D_{2,\max} + D_{3,\max}, & D_{\min} := D_{2,\min} + D_{3,\min} \\
L_{\max} &:= L_{2,\max} + L_{3,\max}, & L_{\min} := L_{2,\min} + L_{3,\min}
\end{align*}
and therefore
\begin{align}
L_{\max} &: D_{\max} \rightarrow \mathcal{L}^2(a_2,a_4), \\
L_{\min} &: D_{\min} \rightarrow \mathcal{L}^2(a_2,a_4).
\end{align}
\end{definition}
The operator $L_{\min}$ is a closed, symmetric, densely defined operator in $\mathcal{L}^2(a_2,a_4)$ and $L_{\max}$, $L_{\min}$ form an adjoint pair, i.e. $L_{\max}^* = L_{\min}$ and $L_{\min}^* = L_{\max}$.
In order to define a self-adjoint extension of $L_{\min}$, we need to introduce the notion of the Lagrange sesquilinear form:
\begin{equation}
[u,v] := uP\overline{v}'-\overline{v}Pu',
\end{equation}
where, at the singular points,
\begin{align}
[u,v](a_i^+) &:= \lim_{\alpha \to a_i^+} [u,v](\alpha), \\
[u,v](a_i^-) &:= \lim_{\alpha \to a_i^-} [u,v](\alpha).
\end{align}
These limits exist and are finite for all $u$, $v \in D_{max}$. If we choose $u$, $v \in D_{\max}$ such that $[u,v](a_i) = 1$ for all the singular points ($a_2^+$, $a_3^-$, $a_3^+$, $a_4^-$), then the extension of $L_{\min}$ defined by the following conditions 
\begin{align}
[\psi,u](a_2^+) &=0 = [\psi,u](a_4^-) \label{sa1} \\
[\psi,u](a_3^-) &= [\psi,u](a_3^+) \label{sa2} \\
[\psi,v](a_3^-) &= [\psi,v](a_3^+) \label{sa3}
\end{align}
is self-adjoint. We refer to \eqref{sa1} as boundary conditions, and to \eqref{sa2} and \eqref{sa3} -- as transmission conditions. The latter connect the two subintervals $(a_2,a_3)$ and $(a_3,a_4)$. Motivated by the conditions mentioned in Section \ref{intro}, we define a self-adjoint extension of $L_{\min}$:
\\
\begin{lemma}\label{lemma1}
The extension $L_S: D(L_S) \rightarrow \mathcal{L}^2(a_2,a_4)$ of $L_{\min}$ to the domain 
\begin{align}
D(L_S) := \{ \psi \in D_{\max}: [\psi,u](a_2^+) &= [\psi,u](a_4^-) = 0, \nonumber \\
[\psi,u](a_3^-) &= [\psi,u](a_3^+), [\psi,v](a_3^-) = [\psi,v](a_3^+) \}
\end{align}
with the following choice of maximal domain functions $u, v \in D_{\max}$
\begin{align}
u(y) &:= 1, \label{u} \\
v(y) &:= \sum_{i=1}^4 \prod_{\substack{j \neq i \\ j \in \{1, \dots, 4 \}}} \frac{1}{a_i-a_j} \ln|y-a_i|, \label{v}
\end{align} 
is self-adjoint.
\end{lemma}
\\

This choice of maximal domain functions gives $[u,v](a_i)=1$ for $i=1,\dots, 4$. The boundary conditions simplify to
 \begin{equation}\label{bc}
\lim_{y \to a_2^+} P(y) \psi'(y) = \lim_{y \to a_4^-} P(y) \psi'(y) = 0.
\end{equation}
For an eigenfunction $\psi$ of $L_S$ this is equivalent to $\psi$ being bounded at $a_2^+$ and $a_4^-$ (because the only possible singularity is of logarithmic type). Let $\phi_1$ and $\phi_2$ be the restrictions of $\psi$ to the intervals $(a_2,a_3)$ and $(a_3,a_4)$, respectively. Since $\psi$ is an eigenfunction, on the corresponding intervals $\phi_1$ and $\phi_2$ are of the form $\phi_i(y) = \phi_{i1} (y) + \phi_{i2} (y) \ln|y-a_3|$. Here, the functions $\phi_{ij}$ are analytic on $(a_2,a_3)$ for $i=1$ and on $(a_3,a_4)$ -- for $i=2$.
Having this, the transmission conditions can be simplified as follows:
\begin{align}
[\psi,u](a_3^+) &= [\psi,u](a_3^-) \nonumber \\ 
\lim_{y \to a_3^+} P(y) \psi'(y) &= \lim_{y \to a_3^-} P(y) \psi'(y) \nonumber \\
\lim_{y \to a_3^-} \phi_{12} (y) &= \lim_{y \to a_3^+} \phi_{22} (y) \label{tc2}.
\end{align}
The condition involving $v$ yields
\begin{align}
[\psi,v](a_3^-)&=[\psi,v](a_3^+) \nonumber  \\
 \lim_{y \to a_3^-}[\psi(y) - v(y) (P \psi')(y)] & =  \lim_{y \to a_3^+}[\psi(y) - v(y) (P \psi')(y)] \label{tc11}\\
\lim_{y \to a_3^-}\phi_{11}(y) &= \lim_{y \to a_3^+}\phi_{21}(y)  \label{tc1}
\end{align}
Note that on each side of \eqref{tc11} the logarithmic terms in $\phi_{i2}$ cancel because of the choice of the constants in $v$. The properties \eqref{bc}, \eqref{tc2} and \eqref{tc1} are the same as the ones found for $f_n$ in Section \ref{intro}. Thus, we have constructed an operator $L_S$ for which close to the points $a_2$, $a_3$ and $a_4$, the eigenfunctions behave in the same way that is expected for the $f_n$'s. 

Close to $a_3$, an eigenfunction $\psi$ is given by
\begin{equation}
\psi(y) = 
\left\{
	\begin{array}{ll}
		 \ell_{11} \sum\limits_{m=0}^{\infty} d^-_m (y-a_3)^m + \ell_{21} \ln|y-a_3| \sum\limits_{m=0}^{\infty} b_m (y-a_3)^m, & y< a_3 \\
		\ell_{12} \sum\limits_{m=0}^{\infty} d^+_m (y-a_3)^m + \ell_{22} \ln|y-a_3| \sum\limits_{m=0}^{\infty} b_m (y-a_3)^m, & y > a_3
	\end{array}
\right. %
\end{equation}
where similarly to \eqref{ff2}, we assume $d_0^-=d_0^+=1$ and $b_0=1$.
The transmission conditions require that
\begin{align}
\ell_{11} &= \ell_{12}, \label{elltc1} \\
\ell_{21} &= \ell_{22}. \label{elltc2}
\end{align}
We can thus express $\psi$ in a sufficiently small neighborhood of $a_3$ as 
\begin{equation}\label{arounda3}
\psi(y) = \ell_{11} + \ell_{21} \ln|y-a_3| \sum_{m=0}^{\infty} b_m (y-a_3)^m + \sum_{m=1}^{\infty} \ell^{\pm}_m (y-a_3)^{m},
\end{equation}
where $\ell^{\pm}_m$ stands for $\ell^+_m = \ell_{11} d_m^+$, when $y >a_3$ and for $\ell^-_m = \ell_{11} d_m^-$, when $y <a_3$.

\section{The spectrum of $L_S$}\label{spec}
In order to prove that the spectrum of the differential self-adjoint operator $L_S$ introduced in Lemma~\ref{lemma1} is discrete, we need to show that for some $z$ in the resolvent set,  $(L_S-zI)^{-1}$ is a compact operator. To do so, it is sufficient to prove that the Green's function $G$ of $L_S-zI$, which for $z$ in the resolvent set exists and is unique, is a function in $\mathcal{L}^2((a_2,a_4)^2)$. This would allow us to conclude that the integral operator $T_G$ with $G$ as its integral kernel is a compact operator from $\mathcal{L}^2(a_2,a_4)$ to $\mathcal{L}^2(a_2,a_4)$, where $T_G$ is equivalent to the inversion of $L_S-zI$. \\ 
\begin{lemma}
The Green's function $G(x,\xi)$ associated with $L_S - i$ is in $\mathcal{L}^2((a_2,a_4)^2)$ and consequently, $(L_S-i)^{-1}: \mathcal{L}^2(a_2,a_4) \rightarrow D(L_S) \subset \mathcal{L}^2(a_2,a_4)$ is a compact operator.
\end{lemma}
\begin{proof}
The self-adjointness of $L_S$ is equivalent to $L_S-i$ being one-to-one and onto (Theorem VIII.3 in \cite{reed}). Moreover, the $a_i$'s are limit-circle points and thus, the deficiency index $d$ equals $4$ (Theorem 13.3.1 in \cite{zettl}). This means that if we do not impose boundary and transmission conditions, there are two linearly independent solutions $p_1$ and $p_2$ of $(L-i) p =0$ on $(a_2,a_3)$ as well as two linearly independent solutions $q_1$ and $q_2$ of $(L-i) q = 0$ on $(a_3,a_4)$. Note that none of these four solutions can be bounded at both of its endpoints because $i$ is not an eigenvalue of the self-adjoint operator $L_{j,S} : D(L_{j,S}) \rightarrow \mathcal{L}^2(I_j)$ with $D(L_{j,S}) = \{ \psi \in D_{j,\max}: \lim_{y \to a_j^+} P(y) \psi'(y) = \lim_{y \to a_{j+1}^-} P(y) \psi'(y) = 0\}$. By taking appropriate combinations, if necessary, we can eliminate the logarithmic singularity at $a_2^+$ of one of the solutions, and at $a_4^-$ -- of another solution. We can thus assume that 
\begin{itemize}
\item on $(a_2,a_3)$: $p_1$ is bounded at $a_2^+$ and logarithmic at $a_3^-$, $p_2$ is logarithmic at both endpoints; \\
\item on $(a_3,a_4)$: $q_1$ is logarithmic at $a_3^+$ and bounded at $a_4^-$, $q_2$ is logarithmic at both endpoints.
\end{itemize}
We next check the restrictions imposed by the transmission conditions at $a_3$. Close to $a_3$, both functions $p_1$ and $q_2$ are of the form \eqref{dof2}. Let $\ell_{11}$, $\ell_{21}$ denote the free parameters in the expression for $p_1$ and $\ell_{12}$, $\ell_{22}$ the ones in $q_2$. These can be chosen such that they satisfy \eqref{elltc1} and \eqref{elltc2}. Thus, there exists a solution $h_1(x)$ on $(a_2,a_4)$ given by
\begin{equation*}
h_1(x) = 
\left\{
	\begin{array}{ll}
		p_1(x) & \mbox{for } x \in (a_2,a_3) \\
		q_2(x) & \mbox{for } x \in (a_3,a_4)
	\end{array}
\right.
\end{equation*}
that is bounded at $a_2^+$ and logarithmic at $a_4^-$. In addition, it is of the form \eqref{arounda3} close to $a_3$, i.e. it is logarithmic at $a_3$ and satisfies the transmission conditions \eqref{tc2}, \eqref{tc1} there. Similarly, with $p_2$ and $q_1$ we can obtain a solution $h_2$ on $(a_2,a_4)$ that satisfies the transmission conditions at $a_3$ and is of $\ln$-$\ln$-bounded-type.
Thus, imposing only the transmission conditions, we obtain two linearly independent solutions of $(L-i) h = 0$ on $(a_2,a_3) \cup (a_3,a_4)$. One of them, $h_1$, is of a bounded-$\ln$-$\ln$-type, and the other one, $h_2$, is of a $\ln$-$\ln$-bounded-type, at the points $a_2^+$, $a_3$, $a_4^-$, respectively. We are now in a position to consider the Green's  function $G(x,\xi)$ of $L_S-i$. Close to $a_3$, we can write the two functions as $h_j(x) = h_{j1}(x) + \ln|x-a_3| h_{j2}(x)$ with continuous functions $h_{j1}$ and $h_{j2}$. By rescaling if necessary, we can assume $h_{12}(a_3) = h_{22}(a_3)$. We construct $G$ from $h_1$ and $h_2$ as follows:
\begin{equation}
G(x,\xi) =
			\left\{
	\begin{array}{ll}
		c_1(\xi) h_1(x)  & \mbox{for } x < \xi \\
		c_2(\xi) h_2(x) & \mbox{for } x > \xi
	\end{array}
\right.
\end{equation}
where $\xi \in (a_2,a_3) \cup (a_3,a_4)$ and the functions $c_1(\xi)$ and $c_2(\xi)$ are chosen such that $G$ is continuous at $x=\xi$ and $\partial G/\partial x$ has a jump discontinuity of $1/P(\xi)$ at $x=\xi$:
\begin{align}
c_1(\xi) h_1(\xi) - c_2(\xi) h_2(\xi) &= 0, \label{G1} \\
c_1(\xi) h_1'(\xi) - c_2(\xi) h_2'(\xi) &= - \frac{1}{P(\xi)} \label{G2}.
\end{align}
In other words, $G$ is the solution of $(L-i) G = \delta$, where $\delta$ is the Dirac delta function. 
For $\xi$ away from $a_3$, $G(x,\xi)$ is continuous in $\xi$ but with logarithmic singularities at $a_2^+$ and $a_4^-$. This can be seen as follows. Consider $\xi$ close to $a_2$. There, we can write 
\begin{equation*}
h_2(\xi) = \tilde{h}_{21}(\xi) + \tilde{h}_{22}(\xi) \ln|\xi-a_2|
\end{equation*}
and, since $h_1$ is bounded close to $a_2^+$, it is of the form \eqref{dof1}, i.e. $h_1(a_2^+) \neq 0$. Let $W_{h_1,h_2}$ denote the Wronskian of $h_1$ and $h_2$, i.e. $W_{h_1,h_2} = h_1 h_2'-h_1' h_2$. For $c_1$ and $c_2$ we obtain
\begin{align}
c_1(\xi) &= \frac{h_2(\xi)}{P(\xi) W_{h_1,h_2}(\xi)}, \\
c_2(\xi) &= \frac{h_1(\xi)}{P(\xi) W_{h_1,h_2}(\xi)}.
\end{align}
The denominator in the above expressions is bounded by
\begin{align*}
P(\xi) (h_1(\xi) h_2'(\xi) - h_2(\xi) h_1'(\xi)) &= \mathcal{O}((\xi-a_2)\ln|\xi-a_2|) + h_1(\xi) \tilde{h}_{22}(\xi) p(\xi), 
\end{align*}
where $p(\xi) = P(\xi)/(\xi-a_2)$ and $h_1(a_2^+) \tilde{h}_{22}(a_2) p(a_2) \neq 0$. Thus, in a neighborhood of $a_2^+$,
\begin{align}
c_1(\xi) &= \mathcal{O}(\ln|\xi-a_2|), \\
c_2(\xi) &= \mathcal{O}(1).
\end{align}
Similarly, since $h_2(a_4^-) \neq 0$, close to $a_4^-$
\begin{align}
c_1(\xi) &= \mathcal{O}(1), \\
c_2(\xi) &= \mathcal{O}(\ln|\xi-a_4|).
\end{align}
For each fixed $\xi \in (a_2,a_3) \cup (a_3,a_4)$, $G(x,\xi)$ as a function in $x$ is continuous on $[a_2,a_3) \cup (a_3,a_4]$ and has a logarithmic singularity at $a_3$, due to the singularities in $h_1(x)$ and $h_2(x)$. It remains to check what happens as $\xi \to a_3$. 
We need to make sure that the functions $c_1(\xi)$ and $c_2(\xi)$ behave in such a way that $G \in \mathcal{L}^2((a_2,a_4)^2)$. Therefore, we derive the asymptotics of $c_1(\xi)$ and $c_2(\xi)$ as $\xi \to a_3^-$. For $\xi = a_3 - \epsilon$ and small $\epsilon > 0$, equation \eqref{G1} becomes
\begin{equation*}
c_1(a_3-\epsilon) h_1(a_3-\epsilon) - c_2(a_3-\epsilon) h_2 (a_3-\epsilon) = 0.
\end{equation*}
Since close to $a_3$, $h_i = h_{i1} + h_{i2} \ln(\epsilon)$ and the $h_{ij}$ are continuous, the ratio $c_1/c_2$ is of the form
\begin{equation*}
\frac{a+b \ln(\epsilon)}{c+d \ln(\epsilon)},
\end{equation*}
where $b$ and $d$ are non-zero (because the logarithmic singularity is present). Thus, the ratio tends to the finite limit $b/d$ as $\epsilon \to 0$. Conditions \eqref{G1} and \eqref{G2} together imply:
\begin{align*}
c_2(a_3-\epsilon) &= \frac{h_1(a_3-\epsilon) }{P(a_3-\epsilon) W_{h_1,h_2}(a_3-\epsilon)} \\
&= \frac{h_{11}(a_3-\epsilon) + h_{12}(a_3-\epsilon) \ln(\epsilon) }{r_1(\epsilon) + \epsilon \cdot r_2(\epsilon)}
\end{align*}
where 
\begin{align*}
r_1(\epsilon) &=  - \frac{1}{\epsilon} P(a_3-\epsilon) \big(h_{21} h_{12} - h_{22} h_{11}\big)(a_3-\epsilon), \\
r_2(\epsilon) &= \mathcal{O}(1),
\end{align*}
and $h_{12}(a_3) \neq 0$. If  $r_1(0) \neq 0$, then $c_2$ is of order $\mathcal{O}(\ln(\epsilon))$, removing a possible obstruction to square integrability of $G$.

Suppose $r_1(0) = 0$, i.e.
\begin{equation*}
h_{21}(a_3) h_{12} (a_3)-h_{22}(a_3) h_{11}(a_3) = 0.
\end{equation*}
This would imply 
\begin{align}
h_{11}(a_3) &= C \cdot h_{21}(a_3), \label{ld1} \\
h_{12}(a_3) &= C \cdot h_{22}(a_3) \label{ld2}
\end{align}
for some constant $C$. By assumption, $h_{12}(a_3) = h_{22}(a_3)$, so that $C=1$. Now if both \eqref{ld1} and \eqref{ld2} hold for $C=1$, the function defined by
\begin{equation*}
h(x) =
\left\{
	\begin{array}{ll}
		h_1(x)  & \mbox{for } x \in (a_2,a_3) \\
		h_2(x) & \mbox{for } x \in (a_3,a_4)
	\end{array}
\right.
\end{equation*}
would be a non-trivial solution of $(L_S-i)h =0$ (fulfilling both boundary and transmission conditions), i.e. $i$ would be an eigenvalue of $L_S$. But this contradicts the self-adjointness of $L_S$. We can thus conclude that $r_1(0) \neq 0$.
This shows that $c_2(a_3-\epsilon) $ is of order $\mathcal{O}(\ln(\epsilon))$ and therefore also $c_2 \cdot \frac{c_1}{c_2} = c_1 = \mathcal{O}(\ln(\epsilon))$. \\
Analogously, we can find the same asymptotics of $c_1(\xi)$ and $c_2(\xi)$ as $\xi \to a_3^+$. \\
Therefore, the properties of the Green's function $G(x,\xi)$ can be summarized as follows:
\begin{itemize}
\item $G( \cdot, \xi)$ has logarithmic singularities at $a_2^+$, $a_3$ and $a_4^-$, \\
\item $G(x, \cdot)$ is of logarithmic singularity at $a_3$, \\
\item away from these singularities $G(x,\xi)$ is continuous in $x$ and $\xi$.
\end{itemize}
Thus,  $G$ is in $\mathcal{L}^2((a_2,a_4)^2)$. Hence, $T_G:\,\mathcal{L}^2(a_2,a_4)\to \mathcal{L}^2(a_2,a_4)$ is a compact Fredholm integral operator. \qquad
\end{proof}
\\

From this we conclude: \\
\begin{proposition}
The operator $L_S$ has only a discrete spectrum, and the associated eigenfunctions are complete in $\mathcal{L}^2(a_2,a_4)$.
\end{proposition}
\\
\begin{proof}
By Theorem VIII.3 in \cite{reed}, the self-adjointness of $L_S$ implies that for the operator $(L_S-i): D(L_S) \rightarrow \mathcal{L}^2(a_2,a_4)$ we have
\begin{align}
\text{Ker}(L_S-i) &= \{ 0\}, \\
\text{Ran}(L_S-i) &= \mathcal{L}^2(a_2,a_4).
\end{align}
Consequently, $(L_S-i)^{-1}:\,\mathcal{L}^2(a_2,a_4)\to D(L_S)$ is one-to-one and onto. Moreover, it is a normal compact operator and thus we get the spectral representation
\begin{equation}
(L_S-i)^{-1} f = \sum_{n=0}^{\infty} \lambda_n \langle f, f_n \rangle f_n,
\end{equation}
where $\{ f_n \}_{n \in \mathbb{N}}$ is a complete orthonormal system in $\mathcal{L}^2(a_2,a_4)$. This can be transformed into the spectral representation for $L_S$:
\begin{align}
L_S f &= \sum_{n=0}^{\infty} (\frac{1}{\lambda_n} + i) \langle f, f_n \rangle f_n.  \qquad
\end{align}
\end{proof}

Clearly, the eigenfunctions $f_n$ of $L_S$ can be chosen to be real-valued. 
The completeness of $\{ f_n \}_{n \in \mathbb{N}}$ is essential for finding the SVD of $H_T$. Another property that will be needed for the SVD is that the spectrum of $L_S$ is simple, i.e. that each eigenvalue has multiplicity $1$. 
\\
\begin{proposition}
The spectrum of $L_S$ is simple.
\end{proposition}
\\
\begin{proof}
From the compactness of $(L_S-i)^{-1}$, we know that each eigenvalue has finite multiplicity. Suppose $f_1$ and $f_2$ are linearly independent eigenfunctions of $L_S$ corresponding to the same eigenvalue $\lambda \in \mathbb{R}$. Then, on all of $(a_2,a_3) \cup (a_3,a_4)$ the following holds
\begin{equation}
f_1 L f_2 - f_2 L f_1 = 0.
\end{equation}
Consequently,
\begin{align*}
0 &= f_1 L f_2 - f_2 L f_1 = f_1 (P f_2')' - f_2 (P f_1')' \\
 &= [f_1,f_2]' .
\end{align*}
Thus, $[f_1,f_2]$ is constant on both $(a_2,a_3)$ and $(a_3,a_4)$. From the boundary conditions that $f_1$ and $f_2$ satisfy, we find that $[f_1,f_2](a_2^+) = 0 = [f_1,f_2](a_4^-)$, which implies $[f_1,f_2] = 0$ on $(a_2,a_3) \cup (a_3,a_4)$.
Since $[f_1,f_2] = P (f_1' f_2 - f_1 f_2')$, we get that
\begin{equation}\label{zerowronskian}
f_1' f_2 - f_1 f_2' = 0 \text{ on } (a_2,a_3) \cup (a_3,a_4).
\end{equation}
The functions $f_1$ and $f_2$ satisfy the transmission conditions at $a_3$. Consequently, they can be written as
\begin{align*}
f_1(x) &= f_{11}(x) + f_{12}(x) \ln|x-a_3|, \\
f_2(x) &= f_{21}(x) + f_{22}(x) \ln|x-a_3|,
\end{align*}
in a neighborhood of $a_3$, where $f_{ij}$ are continuous. Since the one-sided derivatives $f'_{ij}$ are bounded at $a_3$, equation \eqref{zerowronskian} implies
\begin{equation}\label{f1f2}
 \frac{\big( f_{12} f_{21} - f_{11} f_{22} \big) (x)}{x-a_3} + \mathcal{O}(\ln^2|x-a_3|)= 0.
\end{equation}
Note that the terms containing $\ln|x-a_3|/(x-a_3)$ cancel. Taking the limit $x \to a_3$ in \eqref{f1f2}, we obtain
\begin{equation*}
f_{12}(a_3) f_{21}(a_3) - f_{11}(a_3) f_{22}(a_3) = 0.
\end{equation*}
Thus, for some constant $C$:
\begin{equation*}
\left(\begin{array}{c}
f_{11}(a_3) \\
f_{12}(a_3)
\end{array} \right) = C
\left(\begin{array}{c}
f_{21}(a_3) \\
f_{22}(a_3)
\end{array} \right).
\end{equation*}
If we take $f_1$ on $(a_2,a_3)$, then $f_{11}(a_3)$ and $f_{12}(a_3)$ define a singular initial value problem on $(a_3,a_4)$ that is uniquely solvable (Theorem 8.4.1 in \cite{zettl}). Thus, $f_1 = C \cdot f_2$ on $(a_3,a_4)$. Now, on the other hand, by considering $f_1$ on $(a_3,a_4)$, the values $f_{11}(a_3)$ and $f_{12}(a_3)$ define a singular initial value problem on $(a_2,a_3)$ which has a unique solution. Hence, $f_1 = C \cdot f_2$ on $(a_2,a_3) \cup (a_3,a_4)$ in contradiction to our assumption. \qquad
\end{proof}

\section{Singular value decomposition of $H_T$}\label{section:svd}
Having introduced the differential operator $L_S$, we now want to relate it to the truncated Hilbert transform $H_T$. The main result of this section is that the eigenfunctions of $L_S$ fully determine the two families of singular functions of $H_T$. We start by stating the following \\
\begin{proposition}
On the set of eigenfunctions $\{ f_n \}_{n \in \mathbb{N}}$ of $L_S$, the following commutation relation holds:
\begin{equation}\label{comm}
	(H_T L(y,d_y) f_n) (x) = L(x,d_x) (H_T f_n)(x) \quad \text{ for } x \in (a_1,a_2) \cup (a_2,a_3).
\end{equation}
\end{proposition}
{\em Sketch of proof}.
This proof follows the same general idea as the proof of Proposition 2.1 in \cite{kat2}. We therefore provide full details only for those steps where additional care needs to be taken because of the singularity at $a_3$. The steps that are completely analogous to those in the proof of Proposition 2.1 in \cite{kat2} are only sketched here. 

Let $\psi \in \{f_n \}_{n \in \mathbb{N}}$. The boundedness of $\psi$ at $a_2^+$ and $a_4^-$ implies that $P \psi' \to 0$ and $P \psi \to 0$ there. Moreover, the transmission conditions at $a_3$ guarantee that $P \psi'$ is continuous at $a_3$. With these properties, the commutation relation for $x \in (a_1,a_2)$, i.e. where the Hilbert kernel is not singular, can be shown similarly to the proof of Proposition 2.1 in \cite{kat1}.

Next, let $x \in (a_2,a_3)$. The main difference from the proof of Proposition 2.1 in \cite{kat2} is that now the eigenfunctions are not in $C^{\infty}([a_2,a_4])$, but are singular at $a_3$.  However, the fact that we exclude the point $x=a_3$ allows us to always have a neighborhood of $x$ away from $a_3$ on which $\psi$ is bounded. We further note that  $\psi \in C^{\infty}([a_2,a_3) \cup (a_3,a_4])$. Since the Hilbert kernel is singular, we need to use  principal value integration and introduce the following notation: $\ie := [a_2,x-\epsilon] \cup [x + \epsilon,a_4]$. Here $\epsilon > 0$ is so small that $(x-\epsilon,x+2 \epsilon) \subset (a_2,a_3)$, i.e. the $\epsilon$-neighborhood of $x$ is well separated from $a_3$. Then,
\begin{equation*}
\pi (H_T L(y,d_y) \psi) (x) = \lim_{\epsilon \to 0^+} \int_{\ie} \Big[ \frac{(P(y) \psi'(y))'}{y-x} + \frac{2 (y-\sigma)^2 \psi(y)}{y-x} \Big] dy.
\end{equation*}
For the first term under the integral, we integrate by parts twice and plug in the boundary conditions. Again, we use that $P \psi' \to 0$ and $P\psi \to 0$ at $a_2^+$ and $a_4^-$:
\begin{align}
\int\limits_{\ie} \frac{(P(y) \psi'(y))'}{y-x} dy =& - \frac{(P \psi')(x- \epsilon) + (P \psi')(x + \epsilon)}{\epsilon} +  \frac{(P \psi)(x- \epsilon) - (P \psi)(x + \epsilon)}{\epsilon^2} \nonumber \\
& + \int\limits_{\ie} \psi(y) \frac{2 P(y)-P'(y)(y-x)}{(y-x)^3} dy. \label{psipint}
\end{align}
The integral on the right-hand side of \eqref{psipint} can be related to the derivatives of \\ $\int \psi(y)/(y-x) dy$.  
In \cite{kat2} similar relations (cf. eq. (2.7)) were obtained from the Leibniz integral rule, using explicitly that the integrand was continuous. In our case, the function $\psi$ is no longer continuous because of the singularity at $a_3$. We can generalize the argument of \cite{kat2} by invoking the dominated convergence theorem and rewrite the last term in \eqref{psipint} as follows: 
\begin{align*}
&\int\limits_{\ie} \psi(y) \frac{2 P(y) - P'(y)(y-x)}{(y-x)^3} dy = \\
&= P(x) \Big[ \frac{d^2}{dx^2} \int\limits_{\ie} \frac{\psi(y)}{y-x} dy + \frac{\psi'(x-\epsilon) + \psi'(x + \epsilon)}{\epsilon} - \frac{\psi(x-\epsilon) - \psi(x + \epsilon)}{\epsilon^2} \Big] \\
& \quad   + P'(x) \Big[ \frac{d}{dx} \int\limits_{\ie} \frac{\psi(y)}{y-x} dy +  \frac{\psi(x-\epsilon) + \psi(x + \epsilon)}{\epsilon} \Big] \\
& \quad - \int\limits_{\ie} 2 \psi(y) \frac{(y-\sigma)^2-(x-\sigma)^2}{y-x} dy
\end{align*}
Putting all pieces together, we obtain:
\begin{align*}
\pi (H_T L(y,d_y) \psi) &(x) = \\
= \lim_{\epsilon \to 0^+} &\Big\{- \frac{(P \psi')(x-\epsilon)+(P \psi')(x+\epsilon)}{\epsilon} + \frac{(P \psi)(x-\epsilon)-(P \psi)(x+\epsilon)}{\epsilon^2} \\
& \quad + P(x) \big[ \frac{\psi'(x-\epsilon) + \psi'(x + \epsilon)}{\epsilon} - \frac{\psi(x-\epsilon) - \psi(x + \epsilon)}{\epsilon^2} \big] \\
& \quad + P'(x) \frac{\psi(x-\epsilon) + \psi(x + \epsilon)}{\epsilon} + L(x,d_x) \int\limits_{\ie} \frac{\psi(y)}{y-x} dy \Big\}
\end{align*}
The eigenfunction $\psi$ is in $C^{\infty}[a_2,x+2\epsilon)$. 
Following \cite{kat2}, we can thus express the boundary terms in the above equation by Taylor expansions around $x$ and make use of the fact that the boundary terms consist only of odd functions in $\epsilon$.
 The boundary terms are then of the order $\mathcal{O}(\epsilon)$.
 We thus have
 \begin{equation}
 \pi (H_T L(y,d_y) \psi) (x) = \lim_{\epsilon \to 0^+ } L(x,d_x) \int_{\ie} \frac{\psi(y)}{y-x} dy.
 \end{equation}
 Since for $\epsilon > 0$ sufficiently small, $\psi \in C^{\infty}([x-\epsilon,x+\epsilon])$, one can interchange the limit with $L(x,d_x)$ as in \cite{kat2}. \qquad
\endproof
\\

Because the spectrum of $L_S$ is purely discrete, we have thus 
found an orthonormal basis (the eigenfunctions of $L_S$) $\{ f_n \}_{n \in \mathbb{N}}$ of $\mathcal{L}^2(a_2,a_4)$ for which \eqref{comm} holds.
Let us define $g_n := H_T f_n/\| H_T f_n\|_{\mathcal{L}^2(a_1,a_3)}$. Then, in order to obtain the SVD for $H_T$ (with singular functions $f_n$ and $g_n$), it is sufficient to prove that the $g_n$'s form an orthonormal system of $\mathcal{L}^2(a_1,a_3)$ (they will then consequently form an orthonormal basis of $\mathcal{L}^2(a_1,a_3)$, see Proposition \ref{NR}). 

The orthogonality of the $g_n$'s will follow from the commutation relation. Since $f_n$ is an eigenfunction of  $L_S$ for some eigenvalue $\lambda_n$, we obtain
\begin{equation*}
L(x,d_x) g_n(x) = \lambda_n g_n(x), \quad x \in (a_1,a_2) \cup (a_2,a_3).
\end{equation*}
Similarly to $L_S$, we define a new self-adjoint operator that acts on functions supported on $[a_1,a_3]$:
\begin{definition}
Let $\tilde{D}_{\max} := D_{1,\max} + D_{2,\max}$ and $\tilde{L}_{\min} := L_{1,\min} + L_{2,\min}$. The operator $\tilde{L}_S: D(\tilde{L}_S) \to \mathcal{L}^2(a_1,a_3)$ is defined as the self-adjoint extension of $\tilde{L}_{\min}$, where 
\begin{align}
D(\tilde{L}_S) := \{ \psi \in \tilde{D}_{\max}: [\psi,u](a_1^+) &= [\psi,u](a_3^-) = 0, \nonumber \\
[\psi,u](a_2^-) &= [\psi,u](a_2^+), [\psi,v](a_2^-) = [\psi,v](a_2^+) \}
\end{align}
with the maximal domain functions $u, v \in \tilde{D}_{\max}$ as in \eqref{u}, \eqref{v}.
\end{definition}
\\

The intuition then is the following. The function $f_n$ is bounded at $a_2^+$ and logarithmic at $a_3$, where it satisfies the transmission conditions. Consequently, as will be shown below, $g_n$ is bounded at $a_3^-$, logarithmic at $a_2$ and satisfies the corresponding transmission conditions at $a_2$. Clearly, it is also bounded at $a_1^+$. Thus, $g_n$ is an eigenfunction of the self-adjoint operator $\tilde{L}_S$. As a consequence, the $g_n$'s form an orthonormal system. \\
\begin{proposition}
If $L_S f_n = \lambda_n f_n$, then $g_n := H_T f_n/\|H_T f_n \|_{\mathcal{L}^2(a_1,a_3)}$ is an eigenfunction of $\tilde{L}_S$ corresponding to the same eigenvalue
\begin{equation}
\tilde{L}_S g_n = \lambda_n g_n.
\end{equation} 
\end{proposition}
\begin{proof}
First of all, the commutation relation for $f_n$ yields
\begin{align*}
L(x,d_x) (H_T f_n)(x) &= (H_T L(y,d_y) f_n)(x), \\
L(x,d_x) g_n(x) &= \lambda_n g_n(x), \quad x \in (a_1,a_2) \cup (a_2,a_3).
\end{align*}
What remains to be shown is that $g_n$ satisfies the boundary and transmission conditions. Therefore, we consider
$p.v. \int_{a_2}^{a_4} f_n(y)/(y-x) dy$
for $x$ close to $a_1$, $a_2$ and $a_3$. In a neighborhood of $a_1$ away from $[a_2,a_4]$ this function is clearly analytic.
Next, let $x$ be confined to a small neighborhood of $a_2$. Since the discontinuity of $f_n$ is away from $a_2$, we can split the above integral into two -- one that integrates over a right neighborhood of $a_2$ and another one that is an analytic function. The first item in Lemma \ref{propH} then implies that 
\begin{equation}
p.v. \int\limits_{a_2}^{a_4} \frac{f_n(y)}{y-x} dy = \tilde{g}_{n,1}(x) - f_n(a_2^+) \ln|x-a_2|
\end{equation}
where $ \tilde{g}_{n,1}(x)$ is continuous in a neighborhood of $x = a_2$.
Thus, $g_n$ satisfies the transmission conditions \eqref{tc2}, \eqref{tc1}.

It remains to check the behavior of $g_n$ close to $a_3^-$. We first express $f_n$ as 
\begin{equation*}
f_n(y) = f_{n,1}(y) + f_{n,2}(y) \ln|y-a_3|,
\end{equation*}
where both $f_{n,1}$ and $f_{n,2}$ are Lipschitz continuous. Then, in view of Lemma \ref{propH}, both summands on the right-hand side of the equation
\begin{align*}
p.v. \int\limits_{a_2}^{a_4} \frac{f_n(y)}{y-x} dy =& p.v. \int\limits_{a_2}^{a_4} \frac{f_{n,1}(y)}{y-x} dy + p.v. \int\limits_{a_2}^{a_4} \frac{f_{n,2}(y)\ln|y-a_3|}{y-x} dy
\end{align*}
remain bounded as $x$ tends to $a_3$. \qquad
\end{proof}

Since the spectrum of $L_S$ is simple, we can conclude that the $g_n$'s form an orthonormal system and thus the following holds: \\
\begin{theorem}
The eigenfunctions $f_n$ of $L_S$, together with \\ $g_n:=H_T f_n/ \| H_T f_n\|_{\mathcal{L}^2(a_1,a_3)}$ and $\sigma_n := \| H_T f_n \|_{\mathcal{L}^2(a_1,a_3)}$ form the singular value decomposition for $H_T$:
\begin{align}
H_T f_n &= \sigma_n g_n, \\
H_T^* g_n &= \sigma_n f_n.
\end{align}
\end{theorem}

\section{Accumulation points of the singular values of $H_T$}\label{section:acc}
The main result of this section is that $0$ and $1$ are accumulation points of the singular values of  $H_T$. To find this, we first analyze the nullspace and range of $H_T$, which will also prove the ill-posedness of the inversion of $H_T$. First, we need to state the following \\
\begin{lemma}\label{analytic}
If the Hilbert transform of a compactly supported $f \in \mathcal{L}^2(a,b)$ vanishes on an open interval $(c,d)$ disjoint from the object support, then $f=0$ on all of $\mathbb{R}$.
\end{lemma}

{\em Sketch of proof}.
A similar statement (and proof) can be found in \cite{cour}. The main difference is that here we consider a more general class of functions $f$. By dominated convergence, $f \in L^1(a,b)$ implies that for any $z \in \Omega =  \mathbb{C} \backslash ((-\infty,a) \cup (b, \infty))$, the function $g(z) = \int_{a}^b f(x)/(x-z) dx$ is differentiable in a neighborhood of $z$. Thus, $g$ is analytic on $\Omega$. The statement then follows in the same way as Lemma 2.1 in \cite{cour}. \qquad \endproof

With this property of the Hilbert transform, we can obtain results on the nullspace and the range of $H_T$: \\
\begin{proposition}\label{NR}
The operator $H_T: \mathcal{L}^2(a_2,a_4) \rightarrow \mathcal{L}^2(a_1,a_3)$ has a trivial nullspace and dense range that is not all of $\mathcal{L}^2(a_1,a_3)$, i.e.
\begin{align}
\text{Ker}(H_T) &= \{ 0 \}, \label{nht}\\
\text{Ran}(H_T) & \neq \mathcal{L}^2(a_1,a_3),\label{rht1} \\
\overline{\text{Ran}(H_T)} & = \mathcal{L}^2(a_1,a_3). \label{rht2}
\end{align}
\end{proposition}

{\em Proof of \eqref{nht}}.
Suppose $H_T f = 0$. Then 
\begin{align*}
	H \chi_{[a_2,a_4]} f &= 0 \text{ on } (a_1,a_2),
\end{align*}
and by Lemma \ref{analytic}, $f=0$ on all of $[a_2,a_4]$. 
Thus, $f \in \mathcal{L}^2(a_2,a_4)$ can always be uniquely determined from $H_T f$. 

{\em Proof of \eqref{rht1}}.
Take any $g \in \mathcal{L}^2(a_1,a_3)$ that vanishes on $(a_1,a_2)$ and such that $\| g\|_{\mathcal{L}^2(a_1,a_3)} \neq 0$. Suppose $g \in \text{Ran}(H_T)$. By Lemma \ref{analytic}, if $f \in \mathcal{L}^2(a_2,a_4)$ and $H_T f = g$, then $f$ is zero on $[a_2,a_4]$. This implies that $g = 0$ on $(a_1,a_3)$, which contradicts the assumption $\| g \| \neq 0$. 

{\em Proof of \eqref{rht2}}.
The operator $H_T^*$ is also a truncated Hilbert transform with the same general properties. By the above argument, $\text{Ker} (H_T^*) = \{ 0 \}$. \\ Thus, $\text{Ran}(H_T)^{\perp} = \{ 0 \}$. \qquad \endproof

Equation \eqref{rht1} shows the ill-posedness of the problem. It is not true that for every $g \in \mathcal{L}^2(a_1,a_3)$ there is a solution $f$ to the equation $H_T f = g$. Since $\text{Ran}(H_T)$ is dense, the solution need not depend continuously on the data. Thus, our problem violates two properties of Hadamard's well-posedness criteria \cite{engl}. These are the existence of solutions for all data and the continuous dependence of the solution on the data. 
We now turn to the spectrum of $H_T^* H_T$. In what follows, $\| \cdot \|$ denotes the norm associated with $\mathcal{L}^2(\mathbb{R})$, and $\langle \cdot, \cdot \rangle$ denotes the $\mathcal{L}^2(\mathbb{R})$ inner product.
We begin with proving the following \\
\begin{lemma}\label{norm1}
The operator $H_T^* H_T$ has norm equal to $1$.
\end{lemma}
\\
\begin{proof}
From $\| H \| = 1$, we know that $\| H_T^* H_T \| \leq 1$. Since 
\begin{equation*}
\| H_T^* H_T \| = \sup_{\| \psi \| = 1} \| H_T^* H_T \psi \|,
\end{equation*}
finding a sequence $\psi_n$ with $\| \psi_n\| = 1$ and $\| H_T^* H_T \psi_n \| \to 1$ would prove the assertion. 

Take a compactly supported function $\psi \in \mathcal{L}^2([-1,1])$ with $\| \psi \| = 1$ and two vanishing moments, $\int_{-1}^1 \psi(x) dx = 0 = \int_{-1}^1 x \cdot \psi(x) dx$. From this, we define a family of functions, such that the norm is preserved but the supports decrease. More precisely, for $a>2/(a_3-a_2)$, we set
\begin{equation}
\psi_a(x) = \sqrt{a} \psi (a (x-\frac{a_2+a_3}{2})).
\end{equation}
These functions satisfy $\| \psi_a \| = 1$ and supp $\psi_a = [\frac{a_2+a_3}{2}- \frac{1}{a} ,\frac{a_2+a_3}{2} + \frac{1}{a}] \subset [a_2,a_3 ]$. For their Hilbert transforms we obtain
\begin{equation}
(H \psi_a)(x) = \sqrt{a} (H \psi) (a (x-\frac{a_2+a_3}{2})).
\end{equation}
We can write
\begin{align}
H_T^* H_T \psi_a &= - \chi_{[a_2,a_4]} H \chi_{[a_1,a_3]} H \chi_{[a_2,a_4]} \psi_a \nonumber \\
&= - \chi_{[a_2,a_4]} H (I - (I- \chi_{[a_1,a_3]})) H \chi_{[a_2,a_4]} \psi_a \nonumber \\
&= \psi_a + \chi_{[a_2,a_4]} H (I - \chi_{[a_1,a_3]}) H \chi_{[a_2,a_4]} \psi_a; \nonumber \\
(I - H_T^* H_T) \psi_a &=  - \chi_{[a_2,a_4]} H (I - \chi_{[a_1,a_3]}) H \chi_{[a_2,a_4]} \psi_a.
\end{align}
Consider the $\mathcal{L}^2$-norm of the last expression
\begin{align}
\| (I - H_T^* H_T) \psi_a \|^2 &=  \| \chi_{[a_2,a_4]} H (I - \chi_{[a_1,a_3]}) H \chi_{[a_2,a_4]} \psi_a \|^2 \nonumber \\
& \leq  \| (I - \chi_{[a_1,a_3]}) H \chi_{[a_2,a_4]} \psi_a \|^2 \nonumber \\
& = \int_{(-\infty, a_1) \cup (a_3,\infty)} |(H \psi_a)(x)|^2 dx \nonumber \\
&= a \int_{(-\infty, a_1) \cup (a_3,\infty)} |(H \psi)(a (x-\frac{a_2+a_3}{2})|^2 dx \nonumber \\
&= \int_{-\infty}^{a \cdot (a_1-\frac{a_2+a_3}{2})} |H \psi|^2 dy + \int_{a \cdot (a_3-\frac{a_2+a_3}{2})}^{\infty} |H \psi|^2 dy. \label{ma}
\end{align}
Because of the ordering of the $a_i$'s, we have that $a_1-\frac{a_2+a_3}{2} < 0$ and $a_3-\frac{a_2+a_3}{2} > 0$. Since $\psi$ has two vanishing moments, $H \psi$ asymptotically behaves like $1/|y|^3$ and hence, both integrals in \eqref{ma} are of the order $\mathcal{O}(a^{-5})$. Thus, given any $\epsilon > 0$, one can find $a > 2/(a_3-a_2)$ such that
\begin{equation*}
\| (I - H_T^* H_T) \psi_a \| < \epsilon.
\end{equation*}
Consequently,
\begin{align*}
\|  H_T^* H_T \psi_a \| \geq \| \psi_a \| - \| (I - H_T^* H_T) \psi_a \| > 1- \epsilon.
\end{align*}
Therefore
\begin{equation*}
\| H_T^* H_T \psi_a \| \to 1 \text{ as } a \to \infty,
\end{equation*}
which implies that $\| H_T^* H_T \| = 1$. \qquad
\end{proof} 
\\

We are now in a position to prove  \\
\begin{theorem}
The values $0$ and $1$ are accumulation points of the singular values of $H_T$.
\end{theorem}
\begin{proof}
First of all, $0$ and $1$ are both elements of the spectrum $\sigma(H_T^* H_T)$.
For the value $0$, this follows from $\text{Ran} (H_T^* H_T) \subset \text{Ran} (H_T^*) \neq \mathcal{L}^2(a_2,a_4)$. Moreover, since $\| H_T^* H_T \| = 1$ and $H_T^* H_T$ is self-adjoint, the spectral radius is equal to $1$. Thus,
$1 \in \sigma(H_T^* H_T)$.

The second step is to show that $0$ and $1$ are not eigenvalues of $H_T^* H_T$. \\
\textit{$0$ is not an eigenvalue:} If $H_T^* H_T f =0$, then $\| H_T f\|^2 = \langle f, H_T^* H_T f \rangle = 0$. Since $\text{Ker}(H_T) = 0$, this implies $f=0$.
Thus, $\text{Ker}(H_T^* H_T) = \{ 0 \}$. \\
\textit{$1$ is not an eigenvalue:} Suppose there exists a non-vanishing function $f \in \mathcal{L}^2(a_2,a_4)$, such that
\begin{equation*}
- \chi_{[a_2,a_4]} H \chi_{[a_1,a_3]} H \chi_{[a_2,a_4]} f = f.
\end{equation*}
Then,
\begin{align*}
\| H \chi_{[a_2,a_4]} f \|^2 = \| f \|^2 &= - \langle  \chi_{[a_2,a_4]} H \chi_{[a_1,a_3]} H \chi_{[a_2,a_4]} f, f \rangle \\
&= \langle \chi_{[a_1,a_3]} H \chi_{[a_2,a_4]} f, H \chi_{[a_2,a_4]} f \rangle \\
&= \| \chi_{[a_1,a_3]} H \chi_{[a_2,a_4]} f \|^2.
\end{align*}
This implies that $H \chi_{[a_2,a_4]} f$ is identically zero outside $[a_1,a_3]$. By Lemma \ref{analytic}, this implies $f = \chi_{[a_2,a_4]} f =0$, contradicting the assumption $f \not\equiv 0$.
Therefore, $0$ and $1$ are accumulation points of the eigenvalues of $H_T^* H_T$ and consequently, of the singular values of $H_T$. \qquad
\end{proof}

Since the singular values of $H_T$ also accumulate at a point other than zero, the operator $H_T$ is not compact.

\section{Numerical illustration}\label{section:num}
We want to illustrate the properties of the truncated Hilbert transform \textit{with overlap} obtained above for a specific configuration. We choose $a_1=0$, $a_2= 1.5$, $a_3 = 6$ and $a_4 = 7.5$. First, we consider two different discretizations of $H_T$ and calculate the corresponding singular values. We choose the first discretization to be a uniform sampling with $601$ partition points in each of the two intervals $[0,6]$ and $[1.5,7.5]$. Let vectors $X$ and $Y$ denote the partition points of $[0,6]$ and $[1.5,7.5]$ respectively. To overcome the singularity of the Hilbert kernel the vector $X$ is shifted by half of the sample size. The $i$-th components of the two vectors $X$ and $Y$ are given by $X_i = \frac{1}{100} (i + \frac{1}{2})$ and $Y_i = 1.5 + \frac{1}{100} i$; $H_T$ is then discretized as $(H_T)_{i,j} = (1/\pi) (X_i-Y_j)$, $i, j=0, \dots,600$. Figure \ref{fig:sv_can} shows the singular values for the uniform discretization. We see a very sharp transition from $1$ to $0$. 

The second discretization uses orthonormal wavelets with two vanishing moments. Let $\phi$ denote the scaling function. For the discretization we define a finest scale $J=-7$. The scaling functions on $[1.5,7.5]$ are taken to be $\phi_{-7,k}$ for integers $k=192, \dots, 957$, i.e. such that supp $\phi_{-7,k} \subset [1.5,7.5]$. On the interval $[0,6]$ the scaling functions are shifted in the sense that we take them to be $\phi_{-7, \ell+\frac{1}{2}}$ for integers $\ell=0, \dots, 765$, i.e. such that supp $ \phi_{-7,\ell+\frac{1}{2}} \subset [0,6]$. Figure \ref{fig:sv_wav} shows a plot of the singular values of this wavelet discretization of $H_T$. Although the transition is not as sharp as in \ref{fig:sv_can}, the singular values in both cases very clearly accumulate at $0$ and $1$.

\begin{figure}[ht!]
	\begin{center}
	
	\subfigure[Uniform discretization of size $601 \times 601$.]{
	\label{fig:sv_can}
            \includegraphics[width=0.45\textwidth]{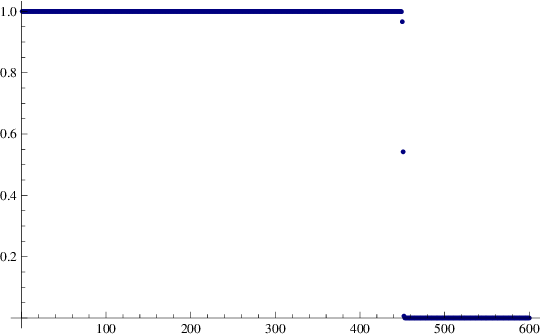}
        }
        \subfigure[Wavelet discretization of size $766 \times 766$.]{
           \label{fig:sv_wav}
                   \includegraphics[width=0.45\textwidth]{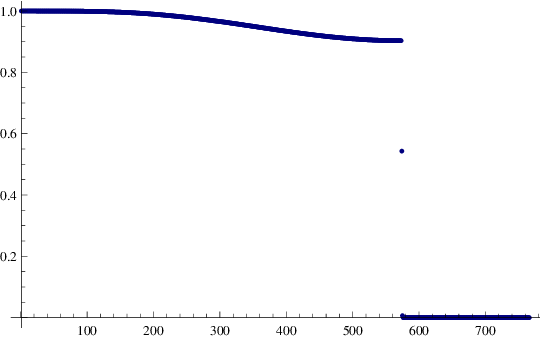}
        }
        \end{center}
  \caption{
 $a_1=0$, $a_2 = 1.5$, $a_3 = 6$, $a_4=7.5$. Singular values of two discretizations of $H_T$.  
  }
  \label{fig:singval}
\end{figure}

Next, we consider the singular functions. Figure \ref{fig:singfunc} shows the singular functions of the uniform discretization for singular values in the transmission region between $0$ and $1$. Figure \ref{fig:singfunc2} illustrates the behavior of singular functions for small singular values. As anticipated, they are bounded at the two endpoints and singular at the point of truncation. Figure \ref{fig:logplot} gives two examples of the close to linear behavior in a log-linear plot of the singular functions. In agreement with the theory in Section \ref{section:svd}, these plots confirm that the singularities are of logarithmic kind.

Based on the numerical experiments conducted, we make the following observations on the behavior of the singular functions and singular values. First, the singular functions in Figures \ref{fig:singfunc} and \ref{fig:singfunc2} have the property that two functions with consecutive indices have the number of zeros differing by $1$. Moreover, the zeros are located only inside one subinterval $I_j$. Furthermore, the plots show that singular functions with zeros within the overlap region correspond to significant singular values, whereas those which have zeros outside the overlap region correspond to small singular values. Finally, we remark that singular functions for small singular values are concentrated \textit{outside} the ROI $I_2=[a_2,a_3]$.

\begin{figure}[ht!]
     \begin{center}

 \subfigure[Singular functions $f_{448}$ and $f_{449}$. ]{
            \label{fig:u01}
            \includegraphics[width=0.45\textwidth]{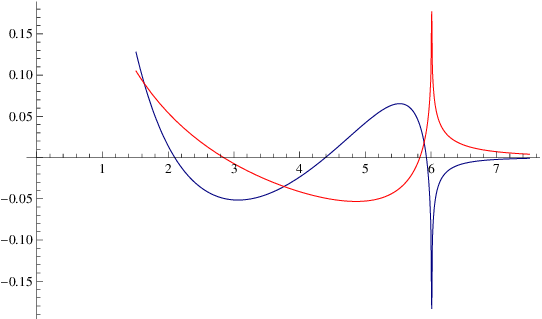}
        }
        \subfigure[Singular functions $f_{450}$ and $f_{451}$.]{
            \label{fig:u23}
            \includegraphics[width=0.45\textwidth]{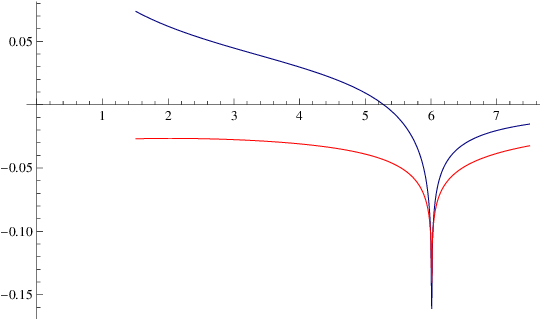}
        }\\ 
        \subfigure[Singular functions $g_{448}$ and $g_{449}$.]{
            \label{fig:v01}
            \includegraphics[width=0.45\textwidth]{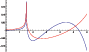}
        }
        \subfigure[Singular functions $g_{450}$ and $g_{451}$.]{
           \label{fig:v23}
           \includegraphics[width=0.45\textwidth]{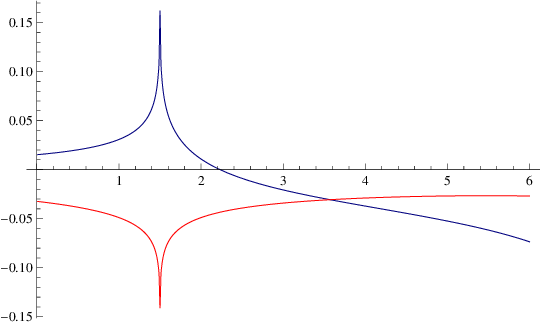}
        }
       
    \end{center}
    \caption{
       Consecutive singular functions for the uniform discretization with $3$, $2$, $1$ and no zeros within the overlap region. The corresponding singular values are $\sigma_{448} = 0.999963$, $\sigma_{449} = 0.998782$, $\sigma_{450} = 0.966192$,  $\sigma_{451} = 0.542071$.
     }
   \label{fig:singfunc}
\end{figure}

\begin{figure}[ht!]
     \begin{center}

        \subfigure[Singular functions $f_{452}$ and $f_{453}$.]{
            \label{fig:uout12}
            \includegraphics[width=0.45\textwidth]{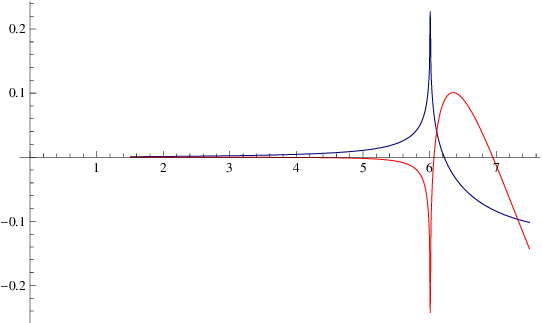}
        }
         \subfigure[Singular functions $f_{454}$ and $f_{455}$.]{
            \label{fig:uout34}
            \includegraphics[width=0.45\textwidth]{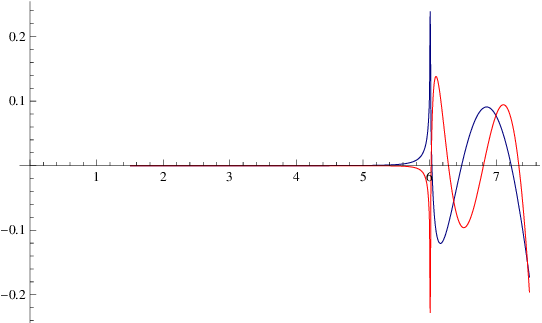}
        }\\
        \subfigure[Singular functions $g_{452}$ and $g_{453}$.]{
           \label{fig:vout12}
           \includegraphics[width=0.45\textwidth]{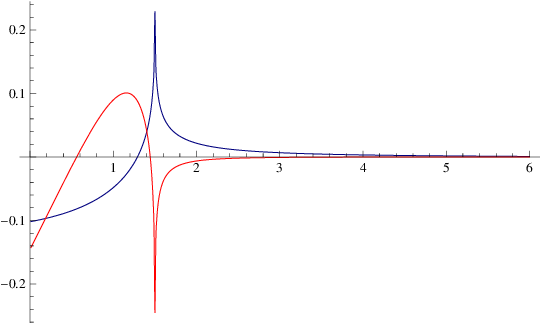}
        }
        \subfigure[Singular functions $g_{454}$ and $g_{455}$.]{
           \label{fig:vout34}
           \includegraphics[width=0.45\textwidth]{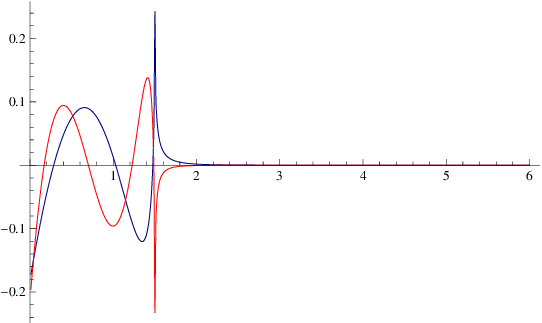}
        }
 \end{center}
    \caption{
       Consecutive singular functions for the uniform discretization with $1$, $2$, $3$ and $4$ zeros outside the overlap region. The corresponding singular values are $\sigma_{452} = 6.29189 \cdot 10^{-3}$, $\sigma_{453} = 2.83533 \cdot 10^{-5}$, $\sigma_{454} = 1.18274 \cdot 10^{-7}$,  $\sigma_{455} =4.83357 \cdot 10^{-10}$.
     }
   \label{fig:singfunc2}
\end{figure}

\begin{figure}[ht!]
     \begin{center}

        \subfigure[]{
            \includegraphics[width=0.45\textwidth]{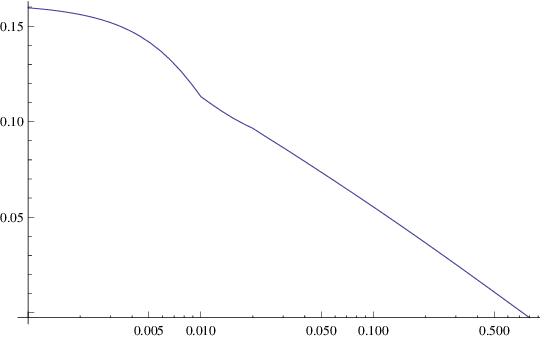}
        }
         \subfigure[]{
            \includegraphics[width=0.45\textwidth]{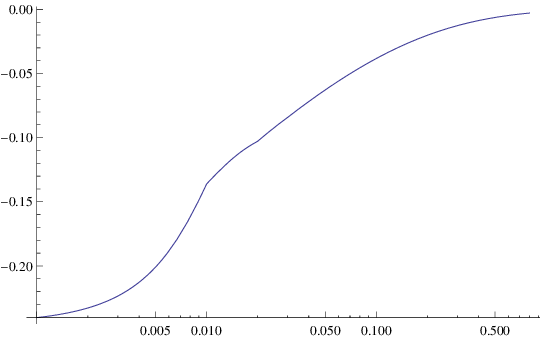}
        }
         \end{center}
    \caption{
       Log-linear plot, i.e. with a logarithmic $x$ scale of the singular functions $g_{450}$ (left) and $g_{453}$ (right) on the interval $[1.5, 2.3].$
      }
   \label{fig:logplot}
\end{figure}

\subsection*{Acknowledgments}
RA was supported by an FWO Ph.D. fellowship and would like to thank Prof. Ingrid Daubechies and Prof. Michel Defrise for their supervision and contribution. Also, RA would like to thank the Department of Mathematics at Duke University for hosting several research stays. AK was supported in part by NSF grants DMS-0806304 and DMS-1211164.

\clearpage
\bibliographystyle{plain}
\bibliography{Spectral-Analysis-Truncated-HT-v2}

\end{document}